\documentclass[a4paper,12pt]{amsart}
\usepackage[utf8]{inputenc}
\usepackage{amssymb, latexsym}
\usepackage{xcolor}
\usepackage{amsthm}
\usepackage{hyperref}

\newtheorem{theorem}{Theorem}[section]

\newtheorem{proposition}[theorem]{Proposition}
\newtheorem{remark}{Remark}
\title[Group of isometries]{Group of isometries of Hilbert ball equipped with the Carath$\mathbf{\acute{E}}$odory metric}
\author{Mukund Madhav Mishra and Rachna Aggarwal }
\address{Department of Mathematics, Hansraj College, University of Delhi, Delhi, India}
\email{mukund.math@gmail.com}
\address{Department of Mathematics, University of Delhi, Delhi, India}
\email{rachna2389@gmail.com}
\date{May 2021}

\begin{document}

\begin{abstract}
In this article, we study the geometry of an infinite dimensional Hyperbolic space. We will consider the group of isometries of the Hilbert ball equipped with the Carath$\acute{e}$odory metric and learn about some special subclasses of this group. We will also find some unitary equivalence condition and compute some cardinalities.
\end{abstract}
\maketitle
\section{Introduction}
Groups of isometries of finite dimensional hyperbolic spaces have been studied by a number of mathematicians; to name a few,  Anderson \cite{JA}, Chen and Greenberg   \cite{CG}, and Parker  \cite{JP}. Hyperbolic spaces can largely be classified into four classes: real, complex, quaternionic hyperbolic spaces and octonionic hyperbolic plane. The respective groups of isometries are $PO(n,1),\,PU(n,1),$\\
$PSp(n,1)$ and $F_{4(-20)}$. Real and complex cases are standard and have been discussed at various places. For example Anderson \cite{JA} and Parker \cite{JP}.  Quaternionic spaces have been studied by Cao and Parker \cite{CP} and Kim and Parker \cite{KP}. For octonionic hyperbolic plane, one may refer to  Baez \cite{BA} and Markham and Parker \cite{MP}. The basic hyperbolic plane model is half plane model usually given by $H^2$ and it is well connected with the famous Poincar$\acute{e}$ disc model in two dimensions \cite{JS}.  The distance in the Poincar$\acute{e}$ disc model is given by the Poincar$\acute{e}$ metric and one of the crucial properties of this metric is that the holomorphic self maps on the Poincar$\acute{e}$ disc satisfy Schwarz-Pick lemma, see \cite[Sec.2.3.5, Theorem 2.3.22]{SK}. Now in an attempt to generalize this lemma to higher dimensions, Carath$\acute{e}$odory and Kobayashi metrics were discovered which formed one of the ways to discuss hyperbolic structure on domains in $\mathbb{C}^n$. It is therefore natural to look for an infinite dimensional counterpart of the finite dimensional hyperbolic spaces. In late 20th century, Franzoni and Vesentini studied infinite dimensional Hilbert ball equipped with the Carath\'eodory metric. A description of the group of isometries of the hyperbolic ball  has been given in \cite{FV}. Motivational sources behind studies carried out in this article are \cite{SA},  \cite{KG} and \cite{KU}.  

In this article, we intend to study dynamical aspects of the group of isometries of infinite dimensional hyperbolic ball. To this end, we consider this group, focus on some special subclasses of this group  and explore their properties.  Sectional detail is as follows. Section 1 is a brief literature survey. Section 2 gives the description of group of holomorphic isometries and its linear representation. In section 3, we will study a class of isometries having a two dimensional reducing subspace, class of normal isometries, self adjoint isometries and involutory isometries. In section 4, we will discuss the condition on an isometry to be unitarily equivalent to its inverse and compute cardinality of the group of isometries, class of self adjoint elements and set of unitary equivalence classes of normal operators on $H \oplus \mathbb{C}$.

\section{The group of isometries}
We begin with the notion of holomorphicity in infinite dimensional set up, give a brief introduction of the Carath$\acute{e}$odory metric and describe the form of corresponding isometries on the infinite dimensional hyperbolic space which has been taken  from \cite{FV}.

 We denote the space of bounded linear operators from a Banach space $V$ to a Banach space $W$ by $B(V,W)$ and space of bounded linear operators on a Hilbert space $H$ by $B(H)$. Let $V$ and $W$ be two complex Banach spaces.  Let $U$ be an open subset of $V$. A mapping $f: U \longrightarrow W$ is called a  holomorphic function if it is Fr$\acute{e}$chet differentiable on $U$, i.e.  for every $a \in U$, there exists $A \in B(V, W)$ such that $\lim\limits_{x \rightarrow a}\dfrac{f(x)-f(a)-A(x-a)}{\|x-a\|}=0$ \cite[Page 68, Definition 147]{HJ}. Let $H$ be an infinite dimensional complex Hilbert space and B be the open unit ball in $H$ equipped with the Carath$\acute{e}$odory metric defined as follows.  Let $D$ be a domain in a complex normed space $E$ and $\Delta$ denote the open unit ball in $\mathbb{C}$ equipped with the Poincar$\acute{e}$ metric $\rho$. Let Hol$(D,\Delta)$ denote the set of all holomorphic mappings $f:D\rightarrow \Delta$. Then
\[C_D(x,y)= \sup_{f}{\rho(f(x),f(y))}\,\,\,\,\,\,\text{for\,\,all}\,\,x,\,y\,\in D\]
defines the Carath$\acute{e}$odory pseudo-distance on $D$. $C_D$ becomes a complete metric on every bounded homogeneous domain in a complex Banach space. This metric is considered a generalization of $\rho$  as it coincides with the latter on  $\Delta$. For a detailed study of this metric, refer  \cite{FV}, \cite{Kob}. A holomorphic map on B is distance decreasing for the Carath$\acute{e}$odory metric. So, every bi-holomorphic surjection is an isometry for this metric. Let $Aut(B)$ denote the group of all  bi-holomorphic surjections on B. This  is the group of isometries we will be working with. A general element of $Aut(B)$ is of the form $F=U\circ f_{x_0}$,  $x_0\in B$, (see \cite[Theorem VI.1.3]{FV}) where $U$ is a unitary operator on $H$ and $f_{x_0} : B_{x_0} \longrightarrow H$ defined by $f_{x_0}(x)=T_{x_0}\left( \dfrac{x-{x_0}}{1-\big<x,x_0\big>} \right)$ is a holomorphic map, $B_{x_0} = \left\{x \in H \,\,:\,\,\|x\|<\dfrac{1}{\|x_0\|}\right\}$. $f_{x_0}$ restricted to B is a bi-holomorphic surjection.  
 $T_{x_0} : H\longrightarrow H$ is a linear map expressed as 
$T_{x_0}(x)=\dfrac{\left<x,x_0 \right>}{1+\sqrt{1-\|x_0\|^2}}x_0+\sqrt{1-\|x_0\|^2}\,\,x.$ $Aut(B)$ is known to act transitively on B.

A linear representation of the elements of $Aut(B)$ is as follows. Let $\mathcal{A}$ be the sesquilinear form on $H\oplus \mathbb{C}$ defined by $\mathcal{A}((x, \lambda),\,(y,\mu))=\left<x,y\right>-\lambda \overline{\mu}$. Consider the collection of all linear operators on $H \oplus \mathbb{C}$ leaving $\mathcal{A}$ invariant. Such linear operators are bounded and injective by definition. Let $G$ be the group consisting of all the bijective linear operators on $H \oplus \mathbb{C}$ leaving $\mathcal{A}$
 invariant. General form of an element of $G$ is $T=\left[ {\begin{array}{cc}
   A & \xi \\
   \left<\cdot,\dfrac{A^*(\xi)}{a}\right> & a \\
  \end{array} } \right]
$ where $A \in B(H),\,\,\xi \in H$ and $a \in \mathbb{C}$ 
satisfying 
\begin{eqnarray*} 
A^{*}A =I+\dfrac{1}{|a|^{2}}\left<\cdot,{A^{*}(\xi)}\right>A^{*}(\xi)
\end{eqnarray*}
and
\begin{eqnarray*} 
|a|^{2}=1+\|\xi\|^{2}.
\end{eqnarray*}

The center $Z(G)$ of $G$ is $\{e^{i \theta}I, \,\,\,\, \theta \in \mathbb{R}\}$,
 see {\cite[Lemma VI.3.4]{FV}}.

 Let $\sigma(T)$ denote the spectrum of a bounded linear operator $T$.
 
\begin{theorem}[\cite{FV} Theorem VI.3.5] The map $\phi : G \rightarrow \text{Aut(B)}$ defined by  $\phi(T)=\widetilde{T}$ is an onto homomorphism where 
$T=\left[ {\begin{array}{cc}
   A & \xi \\
   \left<\cdot,\dfrac{A^*(\xi)}{a}\right> & a \\
  \end{array} } \right]$ and $\widetilde{T}= \dfrac{A(\cdot)+\xi}{\left<\cdot,\frac{A^*(\xi)}{a}\right>+a}$.
\end{theorem}

\begin{proof}
Well definedness and homomorphic nature of $\phi$ can  easily be verified. We will show that $\phi$ is an onto map. For $x_0 \in B$, let $U \circ f_{{x_0}}$ be a general element of $Aut(B)$. We will find preimages of $f_{{x_0}}$ and $U$ separately. For a unitary element $U \in Aut(B)$, $V=\left[{\begin{array}{cc}
U & 0\\
0 & 1\\
\end{array}}\right]$ will work. Let $x_0 \in B\setminus \{0\}$. To find suitable $A,\, \xi$ and $a$ such that $\dfrac{A(x)+\xi}{\left<x,\frac{A^*(\xi)}{a}\right>+a}=f_{{x_0}}(x)$ for all $x \in B$. Suppose such $A, \xi$ and $a$ exist. This gives $\widetilde{T}(0)=\dfrac{\xi}{a}=f_{{x_0}}(0)=-x_0$. Choose $a=\sqrt{1+\|\xi\|^2}$. Then $\|\xi\|^2=(1+\|\xi\|^2)\|x_0\|^2$ and $\|\xi\|=\dfrac{\|x_0\|}{\sqrt{1-\|x_0\|^2}}$. Hence, choose $\xi=-ax_0$. Now, we will choose a suitable self adjoint operator $A$. Construct an operator $S \in B(H)$ such that $S(\xi)=a^2 \xi$ and $S=I$ on $\left<\xi\right>^{\perp}$. Then $S=I+\left<\cdot,\xi\right>\xi$. For, $x=k\xi+y \in H$, for some $k \in \mathbb{C}$ and $y \in \left<\xi\right>^{\perp}$, $S(x)=a^2k \xi+y+k\xi-k\xi=x+(1+\|\xi\|^2-1)k \xi=x+\|\xi\|^2k \xi+\left<y,\xi\right>\xi=x+\left<k\xi,\xi\right>\xi+\left<y,\xi\right>\xi=x+\left<x,\xi\right>\xi$. Also, $S$ is a positive operator such that $\sigma(S)=\{a^2,\,1\}$ making $S$  invertible. Choose $A$ to be the positive square root of $S$. Then $A(\xi)=a\xi$ and $A=I$ on $\left<\xi\right>^{\perp}$. It now follows that with $a$, $\xi$ and $A$ as chosen above, $T_1=\left[ {\begin{array}{cc}
   A & \xi \\
   \left<\cdot,\xi\right> & a \\
  \end{array} } \right] \in G$ satisfies $\widetilde{T_1}(x)=f_{{x_0}}(x)$. Finally, concluding the detailed working, we have for every element $U \circ f_{x_0} \in Aut(B)$ there exists an isometry of the form $V \circ T_1 \in G$   such that  $\phi(V \circ T_1)=\widetilde{V \circ T_1}=U \circ f_{{x_0}}$.
  \qed
 \end{proof}
 
 \begin{remark}
  Since, $\mathrm{ker}(\phi)=Z(G)$, the map $\widetilde{\phi}:G/{Z(G)} \rightarrow Aut(B)$ is an onto isomorphism.
\end{remark}

\section{Some special subclasses of $G$}
We begin by describing a simplified form of elements of $G$.
\begin{proposition}
 A general element of $G$ is of the form $e^{i \theta} \left[ {\begin{array}{cc}
   UA & U(\xi) \\
   \left<\cdot,\xi\right> & a \\
  \end{array} } \right], \,\,\,\theta \in \mathbb{R}$, where $\xi \in H$, $a=\sqrt{1+\|\xi\|^2}$, $U$ is a unitary operator on $H$ and A is a positive operator on $H$ such that $A=I$ on $\left<\xi\right>^{\perp}$ and $A(\xi)=a\xi$.
 \end{proposition}
 \begin{proof}
  From the proof of Theorem 1, we see that that for every element $U \circ f_{x_0} \in Aut(B)$, there exists an isometry of the form $V \circ T_1$ where $V$ is given by  $ \left[ {\begin{array}{cc}
   U & 0 \\
   0 & 1 \\
  \end{array} } \right]$ and $T_1$ is given by $\left[ {\begin{array}{cc}
   A & \xi \\
   \left<\cdot,\xi\right> & a \\
  \end{array} } \right]$. So, $V \circ T_1=\left[ {\begin{array}{cc}
   UA & U(\xi) \\
   \left<\cdot,\xi\right> & a \\
  \end{array} } \right]$ satisfying the desired properties and $\phi(V \circ T_1)=U \circ f_{x_0}$. Also, Remark 1 tells that the pre image of  a general element $U \circ f_{x_{0}}$ of $Aut(B)$ is  the set 
 $\left\{ e^{i \theta} \left[ {\begin{array}{cc}
   UA & U(\xi) \\
   \left<\cdot,\xi\right> & a \\
  \end{array} } \right], \,\,\,\theta \in \mathbb{R}\right\}$ and hence every element of $G$ is of the form  $e^{i \theta} \left[ {\begin{array}{cc}
   UA & U(\xi) \\
   \left<\cdot,\xi\right> & a \\
  \end{array} } \right]$.
 \qed \end{proof}
 From now onwards, $T \in G$ will denote an element of the form $e^{i \theta}\left[ {\begin{array}{cc}
   UA & U(\xi) \\
   \left<\cdot,\xi\right> & a \\
  \end{array} } \right].$\\
  The sesquilinear form $\mathcal{A}$ defined above is identified with the linear operator $A'=\left[ {\begin{array}{cc}
   I & 0 \\
   0 & -1 \\
  \end{array} } \right]$, i.e. $\mathcal{A}((x, \lambda),(y, \mu))=\left<A'(x,\lambda), (y, \mu)\right>$ for all $(x, \lambda),\,(y, \mu) \in H \oplus \mathbb{C}$. 
   This gives
   $T^*=e^{-i \theta}\left[ {\begin{array}{cc}
   (UA)^* & \xi \\
   \left<\cdot,U(\xi)\right> & \overline{a} \\
  \end{array} } \right]$.  Also, for $T$ of the above form, $T^{-1}=e^{-i \theta}\left[ {\begin{array}{cc}
   (UA)^* & -\xi \\
   -\left<\cdot,U(\xi)\right> & \overline{a} \\
  \end{array} } \right]$. Notice that $\overline{a}=a$ as $a \in \mathbb{R}$.\\
  Also observe that in Theorem 1, the operator $V$ is unitary, i.e. $V^*=V^{-1}$ and the  operator $T_1$ is self adjoint, i.e. ${T_1}^*=T_1$. Hence every element of $G$ is a composition of a unitary and a self adjoint element.
  \vspace{5mm}
  
  The following theorem by Hayden and Suffridge gives the fixed point classification of isometries in $Aut(B)$.

  Let $\partial{B}$ denote the boundary of the open unit ball $B$ and $\overline{B}=B \cup \partial 
 {B}$.
 \begin{theorem}[ \cite{HS}]
 If $g \in \text{Aut(B)}$ has no fixed point in \text{B}, then the fixed point set in $\overline{\text{B}}$ consists of one or two points.
 
 \end{theorem}
  We call an isometry in $Aut(B)$ \textit{elliptic} if has a fixed point in B, hyperbolic or parabolic if it is not elliptic and has two or one fixed points on $\partial{B}$ respectively.
  
 We say an element of $G$ is elliptic (resp. hyperbolic or parabolic) if it is in the pre image of an elliptic (resp. hyperbolic or parabolic) isometry of $Aut(B)$ defined by the homomorphism $\phi$.
 
\vspace{5mm}
 Observe that for $x \in \overline{B}$, $x$ is a fixed point for an isometry $\dfrac{UA+ U(\xi)}{\left<\cdot, \xi \right> +a} \in Aut(B)$ if and only if $(x,1)$ is an eigenvector  for the corresponding element in $G$ having eigenvalue $\left<x, \xi\right>+a$. 
  \vspace{5mm}
 
 Here from, we will assume $\xi \neq 0$, unless stated otherwise.
 
 \vspace{5mm}
  We will now investigate a subclass of isometries of $G$ which decomposes $H \oplus \mathbb{C}$ into a two dimensional subspace containing $\mathbb{C}$ and its orthogonal complement.
  
  We know that a general element $T$ of $G$ is of the form $e^{i \theta}\left[{\begin{array}{cc}
 UA & U(\xi)\\
 \left<\cdot, \xi \right> & a\\
 \end{array}}\right]$ where $A$ decomposes $H$ into $\left<\xi\right>$ and $\left<\xi\right>^{\perp}$. If the unitary operator $U$ is chosen in a manner so as to leave $\left<\xi\right>$ invariant, then all the isometries of $G$ comprising of such unitary operators are precisely the ones decomposing $H \oplus \mathbb{C}$ into $\left<\xi\right> \oplus \mathbb{C}$ and its orthogonal complement. 
 
 Thus we have the following.
 \begin{proposition}
 Let $T=\left[{\begin{array}{cc}
 UA & U(\xi)\\
 \left<\cdot,\xi\right> & a\\
 \end{array}}\right] \in G$. Then $T=T_1 \oplus T_2$ where $T_1=T\restriction_{\left<\xi\right> \oplus \mathbb{C}}$ and $T_2=U\restriction_{\left<\xi\right>^{\perp}}$ if and only if $U(\xi)=r\xi$, $|r|=1$.
 \end{proposition}
 \begin{proof}
 Write $H \oplus \mathbb{C}=\left<\xi\right> \oplus \left<\xi\right>^{\perp} \oplus \mathbb{C}$. Let 
 $U(\xi)=r\xi$, $|r|=1$. Then for any element $(\xi, z)\in \left<\xi\right> \oplus \mathbb{C}$, $\left[{\begin{array}{cc}
 UA & U(\xi)\\
 \left<\cdot,\xi\right> & a\\
 \end{array}}\right]\left[{\begin{array}{c}
  \xi\\
 z\\
 \end{array}}\right]=\left[{\begin{array}{c}
 r(a+z)\xi\\
 \|\xi\|^2+az\\
 \end{array}}\right] \in \left<\xi\right> \oplus \mathbb{C}$. Also for $x \in \left<\xi\right>^{\perp}$,  $\left[{\begin{array}{cc}
 UA & U(\xi)\\
 \left<\cdot,\xi\right> & a\\
 \end{array}}\right]\left[{\begin{array}{c}
 x\\
 0\\
 \end{array}}\right]=\left[{\begin{array}{c}
 U(x)\\
 0\\
 \end{array}}\right] \in \left(\left<\xi\right> \oplus \mathbb{C}\right)^{\perp}$. Thereby showing $\left<\xi\right> \oplus \mathbb{C}$ is a reducing subspace for $T$ and $T\restriction_{({\left<\xi\right> \oplus \mathbb{C}})^{\perp}}=U\restriction_{\left<\xi\right>^{\perp}}$.\\
 Conversely, if $T=T_1 \oplus T_2$, then invariance of $\left<\xi\right> \oplus \mathbb{C}$  under $T$ yields invariance of $\left<\xi\right>$ under $U$.
  \end{proof} 
 \begin{proposition}
 Let $T=\left[{\begin{array}{cc}
 UA & r \xi\\
 \left<\cdot,\xi\right> & a\\
 \end{array}}\right] \in G$, $|r|=1$. Then
 \begin{enumerate}
 \item $\sigma(T)=\{\lambda_1,\,\lambda_2\} \cup \sigma(U\restriction_{\left<\xi\right>^{\perp}})$ where $\lambda_1, \,\lambda_2=\dfrac{a(r+1)\pm \sqrt{a^2 (r+1)^2-4r}}{2}$ respectively.
  The eigenspaces corresponding to the eigenvalues $\lambda_1$ and $\lambda_2$ are generated by the eigenvectors $\left(k_1\xi,1\right)$ and $(k_2 \xi,1)$ where
  
   $k_1,\,k_2=\dfrac{a(r-1)\pm \sqrt{a^2(r+1)^2-4r}}{2 \|\xi\|^2}$ respectively.
 \item $|\lambda_1|=\dfrac{1}{|\lambda_2|}$ and $\|k_1\xi\|=\dfrac{1}{\|k_2\xi\|}$.
 \end{enumerate}
 \end{proposition}
 \begin{proof}
 (1) As $T=T_1 \oplus T_2$ (from Proposition 2), $\sigma(T)=\sigma(T_1) \cup \sigma(T_2)$. Clearly, $\sigma(T_2)=\sigma(U \restriction_{\left<\xi\right>^{\perp}})$. Also, $T_1$ is an operator on a 2 dimensional space $\left<\xi\right> \oplus \mathbb{C}$. So, $\lambda$ is an eigenvalue of $T_1$ with eigenvector $(k \xi,1)$ if $\left[{\begin{array}{cc}
 UA & r \xi\\
 \left<\cdot,\xi\right> & a\\
 \end{array}}\right]\left[{\begin{array}{c}
 k\xi\\
 1\\
 \end{array}}\right]=\lambda\left[{\begin{array}{c}
 k\xi\\
 1\\
 \end{array}}\right]$ which gives $r(ak+1)\xi=\lambda k\xi$ and $\lambda=k\|\xi\|^2+a$. Simplification of these two expressions yield $k^2\|\xi\|^2+a(1-r)k-r=0$. Hence $k=\dfrac{a(r-1)\pm \sqrt{a^2(r+1)^2-4r}}{2 \|\xi\|^2}$ (using $a^2=1+\|\xi\|^2$) and \\
 $\lambda=\dfrac{a(r+1)\pm \sqrt{a^2(r+1)^2-4r}}{2}$. Note that $(\xi, 0)$ cannot be an eigenvector for $T_1$.
 
 (2) Observe that $\lambda_1 \lambda_2=r$  and $k_1 k_2=\dfrac{-r}{\|\xi\|^2}$. This gives $|\lambda_1|=\dfrac{1}{|\lambda_2|}$ and $\|k_1\xi\|=\dfrac{1}{\|k_2\xi\|}$ respectively. 
 \qed \end{proof}
 In general, we have 
 \begin{proposition}
 Let $T=\left[{\begin{array}{cc}
 UA & U(\xi)\\
 \left<\cdot,\xi\right> & a\\
 \end{array}}\right] \in G$ and $K$ be a closed subspace of $H$ containing $\xi$ such that $K$ is invariant under $U$. Then $T$ decomposes $H \oplus \mathbb{C}$ into $K \oplus \mathbb{C}$ and its orthogonal complement.
 \end{proposition}
 \begin{proof}
Represent a general element $y \in K$ as $y=k \xi +x$, where $k \in \mathbb{C}$ and $x \in \left<\xi\right>^{\perp} \cap K$. Let $(k\xi+x,z) \in K \oplus \mathbb{C}$, $z \in \mathbb{C}$. Then $\left[{\begin{array}{cc}
 UA & U(\xi)\\
 \left<\cdot,\xi\right> & a\\
 \end{array}}\right]\left[{\begin{array}{c}
 k\xi+x\\
 z\\
 \end{array}}\right]=\left[{\begin{array}{c}
 U((ak+z)\xi+x)\\
 k\|\xi\|^2+az\\
 \end{array}}\right] \in K \oplus \mathbb{C}$ as $K$ is invariant under $U$ thereby making $K \oplus \mathbb{C}$ invariant under $T$. Also,  for $y \in K^{\perp},\,\left[{\begin{array}{cc}
 UA & U(\xi)\\
 \left<\cdot,\xi\right> & a\\
 \end{array}}\right]\left[{\begin{array}{c}
 y\\
 0\\
 \end{array}}\right]=\left[{\begin{array}{c}
 U(y)\\
 0\\
 \end{array}}\right] \in (K \oplus \mathbb{C})^{\perp}$.
 \qed \end{proof}
An operator $T \in {H}$  is said to be \emph{normal} if $T^*T=TT^*$. We are now ready to describe the normal isometries of $G$. Class of normal isometries forms a subpiece of the above defined subclass of $G$. In exact terms, if $\left<\xi\right> \oplus \mathbb{C}$ reduces $T$, then $\left<\xi\right>$ reduces $U$. In particular, the class of all such isometries where $U$ acts as identity on $\left<\xi\right>$ precisely forms the class of normal isometries in $G$. 
  \begin{theorem}[Normal elements of $G$] Let $S \in G$ be given by $e^{i \theta}\left[ {\begin{array}{cc}
   UA & U(\xi) \\
   \left<\cdot,\xi\right> & a \\
  \end{array} } \right]$, $\theta \in \mathbb{R}$. Then
  
  \begin{enumerate}
 \item $S$ is normal if and only if $S=e^{i \theta}\left[ {\begin{array}{cc}
   UA & \xi \\
   \left<\cdot,\xi\right> & a \\
  \end{array} } \right]$, i.e. $U(\xi)=\xi$.
  
  \item  $S$ is unitary if and only if  $S=e^{i \theta} \left[ {\begin{array}{cc}
   U & 0 \\
   0 & 1 \\
  \end{array} } \right]$, i.e. $\xi=0$.
  \item If $S$ is normal, then $\sigma(S)=\{a \pm \|\xi\| \} \cup \sigma\left(U \big|_{\left<\xi\right>^{\perp}}\right)$, where $a \pm \|\xi\|$ are both positive, non-unit modulus and inverses of each other.  Eigenspaces corresponding to the eigenvalues $a \pm \|\xi\|$ are spanned by the eigenvectors  $\left(\pm \dfrac{\xi}{\|\xi\|},1\right)$ respectively.
  \item Normal isometries are hyperbolic in nature.
  \end{enumerate}
  \end{theorem}
  \begin{proof}
(1)  For an element $S \in G$, direct computation shows that $SS^*=S^*S$ if and only if $U(\xi)=\xi$.

(2)  $S^*=S^{-1}$ if and only if $\xi=0$ is straightforward. Notice that if $\xi=0$, then $A$ reduces to $I$ on $H$.

(3) follows from Proposition 3 for $r=1$.

 (4)  $S$ has two fixed points $\left\{\pm \dfrac{\xi}{\|\xi\|}\right\}$ lying on the unit sphere. Also $S$ can't be elliptic as it does not have an eigenvector of the form $(x,1),\, \|x\|<1$, for eigenvectors of S either come from $\left<\xi\right> \oplus \mathbb{C}$ or $(\left<\xi\right> \oplus \mathbb{C})^{\perp}$ and $a \pm ||\xi|| \neq 1$.
  \qed \end{proof}
 Next we explore the  self adjoint elements of $G$. Prior to that, let us make a simple observation. In the construction of $T$, $A=I$ on $\left<\xi\right>^{\perp}$. This implies that every unitary operator $U$ which decomposes $H$ into $\left<\xi\right>$ and $\left<\xi\right>^{\perp}$  commutes with $A$. 
\begin{proposition}[Self adjoint elements of $G$]
Let $T$ be a self adjoint element in $G$. Then
\begin{enumerate}
\item $T$ is of the form $\pm\left[{\begin{array}{cc}
UA & \xi\\
\left<\cdot,\xi\right> & a\\
\end{array}}\right]$, where $U$ is a unitary  and involutory operator.
\item  $\sigma(T)=\{a \pm \|\xi\|\} \cup \{\pm 1\}$.
\end{enumerate}
\end{proposition}
 \begin{proof}
(1)  $T \in G$ is self adjoint if and only if $T=T^*$, i.e  $e^{i \theta}\left[{\begin{array}{cc}
UA & U(\xi)\\
\left<\cdot,\xi\right> & a\\
\end{array}}\right]=e^{-i\theta}\left[{\begin{array}{cc}
(UA)^* & \xi\\
\left<\cdot,U(\xi)\right> & a\\
\end{array}}\right]$ which gives $UA$ is self adjoint, $U(\xi)=\xi$ and $\theta=n \pi$, $n \in \mathbb{Z}$. Now, $UA$ is self adjoint, $A$ is self adjoint and $U$ commutes with $A$ (by the above observation) gives $U$ is self adjoint. Hence $U$ is involutory.
 
(2) Clearly, $T=T_1 \oplus T_2,\,  \sigma(T_2)= \sigma(U \restriction_{\left<\xi\right>^{\perp}})=\{\pm 1\}$.
 \qed \end{proof}
 
 Let us now discuss the form of involutory elements of $G$. It is easy to see that unitary and involutory elements of $G$ are of the form $\left[{\begin{array}{cc}
 V & 0\\
 0 & \pm 1\\
 \end{array}}\right]$ where $V$ is unitary and involutory.
 
 \begin{proposition}[Involutory elements of $G$]
Let $T$ be an involutory element in $G$. Then
 \begin{enumerate}
     \item $T=\pm \left[ {\begin{array}{cc}
   UA & -\xi \\
   \left<\cdot,\xi\right> & a \\
  \end{array} } \right]$ where $U$ is involutory.
  \item $T=T_1 \oplus T_2$ where $T_1=T\restriction_{\left<\xi\right> \oplus \mathbb{C}}$ and $T_2=U\restriction_{{\left<\xi\right>}^{\perp}}$.
  \item $\sigma(T)=\{\pm1\}$, where $\sigma(T_1)=\pm 1$ with eigenspace spanned by\\ $\left(\dfrac{-a \pm 1}{\|\xi\|^2}\,\xi\,,1\right)$.
  \item Involutory elements are elliptic in nature. 
 \end{enumerate}
 \end{proposition}
 
 \begin{proof}
 (1) $T$ is involutory if and only if $T=T^{-1}$. On comparing respective entries of $T$ and $T^{-1}$, we get $UA$ is self adjoint,  $U(\xi)=-\xi$ and $\theta =n \pi$. Following the same argument as in part (1) of the previous proposition, we get $U$ is involutory.
 
(2) and (3) follow from Proposition 3 for $r=-1$.

(4) Since $\dfrac{|-a+1|}{\|\xi\|}=\dfrac{a-1}{\|\xi\|}<1$, $T$ is elliptic in nature.
 \qed \end{proof}

 \section{Unitarily equivalence condition and cardinality}
Next we will learn about the condition on an isometry to be unitarily equivalent to its inverse.

\begin{proposition}
An isometry $T=e^{i\theta}\left[{\begin{array}{cc}
UA & U(\xi)\\
\left<\cdot,\xi\right> & a\\
\end{array}}\right]$ and its inverse are unitarily equivalent if and only if there exists a unitary operator $V$ such that $(UA)^*=V^{-1}UAV$, $V(\xi)=V^{-1}(\xi)=-U(\xi)$ and $\theta=n\pi$.
\end{proposition}
\begin{proof}
An isometry $T$ and its inverse are unitarily equivalent to each other if and only if there exists a unitary operator $V$ satisfying 
\[\left[{\begin{array}{cc}
V^{-1} & 0\\
0 & 1\\
\end{array}}\right] \left\{e^{i\theta}\left[{\begin{array}{cc}
UA & U(\xi)\\
\left<\cdot,\xi\right> & a\\
\end{array}}\right]\right\}\left[{\begin{array}{cc}
V & 0\\
0 & 1\\
\end{array}}\right]=e^{-i \theta}\left[{\begin{array}{cc}
(UA)^* & -\xi\\
-\left<\cdot,U(\xi)\right> & a\\
\end{array}}\right]\]
i.e.
\[e^{i \theta}\left[{\begin{array}{cc}
V^{-1}UAV & V^{-1}U(\xi)\\
\left<\cdot,V^{-1}(\xi)\right> & a\\
\end{array}}\right]=e^{-i\theta}\left[{\begin{array}{cc}
(UA)^* & -\xi\\
-\left<\cdot,U(\xi)\right> & a\\
\end{array}}\right]\]
if and only if $(UA)^*=V^{-1}UAV$, $V(\xi)=V^{-1}(\xi)=-U(\xi)$ and $\theta=n\pi$.
\qed \end{proof}
Example of an operator which is unitarily equivalent to its inverse is $\left[{\begin{array}{cc}
A & \xi\\
\left<\cdot,\xi\right> & a\\
\end{array}}\right]$ where the conjugating operator $V$ can be taken $-I$. Notice that the way $A$ is defined, $A$ clearly commutes with $V$.
\vspace{5mm}

In finite dimensions, the number of conjugacy classes of centralizers (called $z$-classes) in a group of isometries uses to have implications on the dynamical types of the action of the group on the underlying space, see  [\cite{SA}, \cite{KG}, \cite{KU}]. In infinite dimensional setup, this number should better be represented by cardinality of the set of $z$-classes. Computation of the size (in terms of cardinality) of the set of conjugacy classes happens to be the initial attempt towards the computation of number of $z$-classes. In what follows, we use the spectral decomposition of normal operators to find cardinality of the set of conjugacy classes of normal operators.\\

\textbf{In what follows, we will assume $H$ to be a separable Hilbert space.}\\

 We know that $|H|=\mathfrak{c}$. Let $U(H)$ denote the group of unitary operators on $H$ and $G_s$ denote the collection of self adjoint isometries in $G$. Then 
 
 \begin{proposition}
  $|U(H)|=|G|=|G_s|=\mathfrak{c}$. 
  \end{proposition}
  \begin{proof}
  As  $e^{i \theta}I \in U(H)$, $\theta \in \mathbb{R}$,   we have $|U(H)| \geq \mathfrak{c}$. Also, $|U(H)| \leq |L(H)|$, where $L(H)$ is the space of linear operators on $H$. Now, each linear operator has an infinite matrix representation. This gives $|L(H)|$ is same as the cardinality of the set of all complex sequences which is same as $\mathfrak{c}$. Hence $|U(H)|=\mathfrak{c}$.
  
  Also, by the construction of elements of $G$, it can be seen that $|G|=|H|\times |U(H)| \times |\mathbb{R}|=\mathfrak{c} \times \mathfrak{c} \times \mathfrak{c}=\mathfrak{c}$.
  
  For the collection $G_s$, elements of the form $\left[{\begin{array}{cc}
  A & \xi\\
  \left<\cdot,\xi\right> & a\\
  \end{array}}\right]$ where $\xi \in H$ lie in $G_s$ and cardinality of such elements is $|H|=\mathfrak{c}$. This gives $|G_s| \geq \mathfrak{c}$. Also, $|G_s| < |G|=\mathfrak{c}$. So $|G_s|=\mathfrak{c}$.
  \qed \end{proof}
  
  Next we will find out the cardinality of the set of unitary equivalence classes of normal operators on $H$. For this, we will use the following result from Conway \cite[Chapter 9, Corollary 10.12]{JBC}.
  \begin{theorem}
  If $N$ is a normal operator on a separable Hilbert space $H$ with scalar-valued spectral measure $\mu$, then there is a decreasing sequence $\{\Delta_n\}$ of Borel subsets of $\sigma(N)$ such that $\Delta_1=\sigma(N)$ and 
  \[N \cong N_{\mu} \oplus N_{\mu\restriction_{\Delta_2}} \oplus N_{\mu\restriction_{\Delta_3}} \oplus ...\, .\]
 \end{theorem}
  Let us introduce some quick notations for the sake of next proposition.
  
   Let the set of all normal operators be denoted by $\mathcal{N}$, set of all unitary equivalence classes of normal operators by $\mathcal{N_{U}}$, set of all compact subsets of $\mathbb{C}$ by $\mathcal{K}$, Borel $\sigma$-algebra of a set $A \subset \mathbb{C}$ by $\mathcal{B}(A)$ and the set of all sequences of all the Borel subsets of a set $A$ by $\mathcal{B}_s(A)$.
   
   The above theorem gives $ |\mathcal{N}| \leq |\underset{K \in \mathcal{K}}{\bigcup} \mathcal{B}_{s}(K)| $.
  \begin{proposition}
  $|\mathcal{N_{U}}|=\mathfrak{c}$.
  \end{proposition}
  \begin{proof}
   Let $K \subset \mathbb{C}$ be a compact set. Then there exists a normal operator with $K$ as its spectrum. Also, unitarily equivalent normal operators have the same spectrum. This gives $|\mathcal{N_{U}}| \geq |\mathcal{K}|$. We know that $|\mathcal{B}(\mathbb{C})|=\mathfrak{c}$ (\cite[page no. 53]{WR}). 
   Now, every singleton is a compact set in $\mathbb{C}$ and every compact set is a Borel subset of $\mathbb{C}$ giving $|\mathcal{K}|    =\mathfrak{c}$ and hence $|\mathcal{N_{U}}|\geq \mathfrak{c}$.
   
  We will now show that $|\mathcal{N_{U}}| \leq \mathfrak{c}$. $|\mathcal{B}(K)| \leq |\mathcal{B}(\mathbb{C})|=\mathfrak{c}$ which gives $|\mathcal{B}_s(K)| \leq |\mathcal{B}_s(\mathbb{C})|=\mathfrak{c}$, therefore $|\underset{K \in \mathcal{K}}{\bigcup} \mathcal{B}_{s}(K)| \leq \mathfrak{c} \times \mathfrak{c}=\mathfrak{c}$. Finally, $|\mathcal{N_{U}}| \leq |\mathcal{N}|\leq |\underset{K \in \mathcal{K}}{\bigcup} \mathcal{B}_{s}(K)| \leq \mathfrak{c}$. Hence the result.
  \qed \end{proof}
\textbf{Acknowledgement}
The research is supported by Council of Scientic and Industrial Research, India (File no. 09/045(1668)/2019-EMR-I).

   \bibliographystyle{amsplain}

\end{document}